\documentclass[reqno,12pt,letterpaper]{amsart}
\usepackage{sdmacros}

\newcommand{\RR}{{\mathbb R}}

\title[Ruelle zeta function]%
{Ruelle zeta function at zero for surfaces}
\author{Semyon Dyatlov}
\email{dyatlov@math.mit.edu}
\address{Department of Mathematics, Massachusetts Institute of Technology,
 Cambridge, MA 02139}
\author{Maciej Zworski}
\email{zworski@math.berkeley.edu}
\address{Department of Mathematics, University of California, Berkeley, 
CA 94720}
\begin{document}

\begin{abstract}
We show that the Ruelle zeta function for
a negatively curved oriented surface vanishes at zero to the order 
given by the absolute value of the Euler characteristic.
This result was previously known only in constant curvature.
\end{abstract}

\maketitle

%%%%%%%%%%%%%%%%%%%%%%%%%%%%%%%%%%%%%%%%%%%%%%%%%%%%%%%%%%%%%%%%%%%%%%%%%%%%%%%%
%%%%%%%%%%%%%%%%%%%%%%%%%%%%%%%%%%%%%%%%%%%%%%%%%%%%%%%%%%%%%%%%%%%%%%%%%%%%%%%%

\section{introduction}

Let $ (\Sigma , g ) $ be a compact oriented Riemannian surface of negative curvature
and denote by $ \mathcal G $ the set of primitive closed geodesics on 
$ \Sigma $ (counted with multiplicity). For $ \gamma \in \mathcal G $ denote by $ \ell_\gamma $ its length. The Ruelle zeta
function \cite{Ru1} is defined
by the analogy with the Riemann zeta function,
$ \zeta ( s ) = \prod_p ( 1 - p^{-s} )^{-1} $, replacing primes 
$ p $ by primitive closed geodesics:
\begin{equation}
  \label{eq:defRuz}
\zeta_R ( s ) := \prod_{ \gamma \in \mathcal G } ( 1 - e^{ -s  \ell_\gamma } ) .
\end{equation}
The infinite product converges for $ \Re s \gg 1 $ and the meromorphic 
continuation of~$ \zeta_R $ to $\mathbb C$ has been a subject of extensive study.

Thanks to the Selberg trace formula the order of vanishing of $ \zeta_R ( s ) $
at $ 0$ has been known for a long time 
in the case of {\em constant curvature} and it is given by $ - \chi ( \Sigma ) $ where $ \chi ( \Sigma ) $ is the Euler characteristic. We show that the 
same result remains true for {\em any} negatively curved oriented surface:

\vspace{0.04in}

%%%%%%%%%%%%%%%%%%%%%%%%%%%%%%%%%%%%%%%%%%%%%%%%%%%%%%%%%%%%%%%%%%%%%%%%%%%%%%%%
\noindent
{\bf Theorem.} \emph{ Let $ \zeta_R ( s ) $ be the Ruelle zeta
function for an oriented negatively curved $C^\infty$ Riemannian surface
$ ( \Sigma, g ) $ and 
let $ \chi ( \Sigma ) $ be its
Euler characteristic.
Then $s^{\chi(\Sigma)}\zeta_R(s)$ is holomorphic at $s=0$ and
\begin{equation}
\label{eq:t1}
s^{\chi(\Sigma)}\zeta_R(s)|_{s=0}\neq 0.
\end{equation}
}
%%%%%%%%%%%%%%%%%%%%%%%%%%%%%%%%%%%%%%%%%%%%%%%%%%%%%%%%%%%%%%%%%%%%%%%%%%%%%%%%

\vspace{-0.05in}

\Remarks
1. The condition that the surface is $ C^\infty $ can be replaced by $ C^k $ for a sufficiently 
large $ k $~-- that is an automatic consequence of our microlocal methods. 

\noindent
2. As was pointed out to us by Yuya Takeuchi, 
our proof gives a stronger result in which the cosphere 
bundle $ S^* \Sigma = \{ ( x , \xi ) \in T^* \Sigma :
| \xi|_g = 1 \} $ is replaced by a connected contact 3-manifold~$ M$ whose
contact flow has the Anosov property 
with orientable stable and unstable bundles (see \S\S \ref{s:resonances},\ref{s:gf}).
If $ \mathbf b_1 ( M ) $ denotes the first Betti number of $ M $
(see \eqref{e:Betti}) then $ s^{ 2 - \mathbf b_1 ( M ) } \zeta_R ( s ) $
is holomorphic at $ 0 $ and 
\begin{equation}
\label{eq:b1M}   
s^{ 2 - \mathbf b_1 ( M ) } \zeta_R ( s ) |_{s=0} \neq 0 .
\end{equation}
Theorem above follows from the fact that for 
negatively curved surfaces $ 2 - \mathbf b_1 ( S^* \Sigma ) = \chi( \Sigma ) $ (see
Lemma \ref{l:topo} for the review of this standard fact).  For the 
existence of contact Anosov flows on 3-manifolds which do not arise from 
geodesic flows see \cite{HF}.

\noindent
3. Our result implies that for a negatively curved connected oriented Riemannian surface,
its length spectrum (that is, lengths of closed geodesics counted with multiplicity)
determines its genus. This appears to be a previously unknown inverse result~-- we refer
the reader to reviews~\cite{RBM,Wilkinson,Zelditch} for more information.

For $ ( \Sigma , g) $ of constant curvature the meromorphy of $ \zeta_R $ follows
from its relation to the Selberg zeta function:
$$
\zeta_S ( s ) := \prod_{ \gamma \in \mathcal G } \prod_{ m=0}^\infty
( 1 - e^{- ( m + s ) \ell_\gamma } ) , \quad
\zeta_R ( s ) = \frac{ \zeta_S ( s  ) }{ \zeta_S ( s + 1  ) } ,
$$
see for instance \cite[Theorem~5]{Mark} for a self-contained presentation. 
In this case the behaviour at $ s = 0 $ was analysed by Fried \cite[Corollary 2]{Fr1}
who showed that 
\begin{equation}
\label{eq:fried}    \zeta_R ( s ) = \pm ( 2 \pi s )^{ |\chi ( \Sigma )| } ( 1 + \mathcal O ( s ) ), \end{equation}
where $ \chi ( \Sigma ) $ is the Euler characteristic of $ \Sigma $. A far reaching
generalization of this result to locally symmetric manifolds 
has recently been provided by Shen 
\cite[Theorem 4.1]{Sh} following earlier contributions
by Bismut~\cite{Bismut}, Fried~\cite{Fr-Anator}, and Moskovici--Stanton~\cite{Moskovici}.

For real analytic metrics the meromorphic continuation of $ \zeta_R ( s ) $ 
is more recent and 
follows from results of Rugh \cite{Ru} and Fried \cite{Fr2} 
proved twenty years ago. 
In the $ C^\infty $ case (or $ C^k $ for sufficiently large $ k $) 
that meromorphic continuation is very recent.
For Anosov flows on compact manifolds it was first established 
by Giulietti--Liverani--Pollicott~\cite{glp} and then by Dyatlov--Zworski~\cite{DZ}.
See these papers and 
\cite[Chapter 4]{ZS} for more references and for background information.
 { Here we only mention two particularly 
relevant contributions:  \cite{DG} where the more complicated non-compact 
case is considered 
%essentially settling the original conjecture of Smale \cite{Sm} 
and \cite{River} where microlocal methods are used to 
describe the correlation function of a Morse--Smale gradient flow.}

The value at zero of the dynamical zeta function for
certain two-dimensional hyperbolic open billiards was computed by Morita~\cite{Morita}
using Markov partitions. It is possible that similar methods could work in our
setting  because of the better regularity of stable/unstable foliations
in dimensions 2. However, our spectral approach is more direct and,
as it does not rely on regularity of the stable/unstable foliations,
can be applied in higher dimensions.

The first step of our proof is the standard factorization of $ \zeta_R $ 
which shows that the multiplicity of the zero (or pole) of $ \zeta_R $ 
can be computed from the multiplicities of Pollicott--Ruelle resonances
of the generator of the flow, $ X$,  acting on differential forms~-- see~\S\S\ref{s:resonances},\ref{s:zeta}. The resonances are defined as eigenvalues
of $ X $ acting on microlocally weighted spaces 
-- see~\eqref{e:mercon} which we recall from the 
work of Faure--Sj\"ostrand \cite{fa-sj} and \cite{DZ}.
The key fact, essentially from \cite{fa-sj} -- see \cite[Lemma 5.1]{DGF} 
and Lemma \ref{l:res-states} below~-- is that 
the generalized eigenvalue problem is equivalent to solving the equation
$ ( X + s )^k u = 0 $ under a {\em wavefront set condition}. 
We should stress that the origins of this method lie
in the works on anisotropic Banach spaces by
Baladi~\cite{bal},
Baladi--Tsujii~\cite{bats},
Blank--Keller--Liverani~\cite{b-k-l},
Butterley--Liverani~\cite{but-liv},
Gou\"ezel--Liverani~\cite{go-liv}, and
Liverani~\cite{liverani,liverani2}.

Hence we need to show that the multiplicities of generalized eigenvalues at $s=0$
are the same as in the 
case of constant curvature surfaces (for detailed analysis of
Pollicott--Ruelle resonances in that case we refer to \cite{DGF} and \cite{g-w}).
For functions and 2-forms that is straightforward. For 1-forms 
the dimension of the eigenspace turns out to be easily computable using
the behaviour of $ ( X + s)^{-1} $ near $ 0 $ acting on functions
and is given by the first Betti number. That
is done in \S \ref{s:1fo} and it works for any 
contact Anosov flow on a 3-manifold.
In the case of orientable stable and
unstable manifolds that gives holomorphy of 
$ s^{ 2 - \mathbf b_1 ( M ) } \zeta ( s ) $ at $ s = 0$. 

To show \eqref{eq:b1M}, that is to see that the order of vanishing is exactly $ 2 - \mathbf b_1 ( M )  $, we need
to show that zero is a semisimple eigenvalue, that is its algebraic and geometric
multiplicities are equal. The key ingredient
 is a regularity result given in Lemma~\ref{l:mystery}. It holds
for any Anosov flow preserving a smooth density and could be of independent
interest.

\medskip\noindent\textbf{Acknowledgements.}
We gratefully acknowledge partial support by a Clay Research Fellowship
(SD) and by the National Science Foundation grant DMS-1500852 (MZ). 
We would also like to thank Richard Melrose for suggesting the proof of
Lemma \ref{l:de-rham}, Fr\'ederic Naud for informing us of 
reference \cite{Morita} and the anonymous referee for helpful comments. We are particularly grateful to Yuya Takeuchi
for pointing out that a topological assumption made in an earlier version was
unnecessary -- that lead to the stronger result described in Remark 2 above.

%%%%%%%%%%%%%%%%%%%%%%%%%%%%%%%%%%%%%%%%%%%%%%%%%%%%%%%%%%%%%%%%%%%%%%%%%%%%%%%%
\section{Ingredients}
\label{gepr}

%%%%%%%%%%%%%%%%%%%%%%%%%%%%%%%%%%%%%%%%%%%%%%%%%%%%%%%%%%%%%%%%%%%%%%%%%%%%%%%%
\subsection{Microlocal analysis}

Our proofs rely on microlocal analysis, and we briefly describe microlocal tools
used in this paper providing detailed references to~\cite{ho1,ho3,ev-zw,DZ} and
\cite[Appendix E]{res}.

Let $M$ be a compact smooth manifold and $\mathcal E,\mathcal F$ smooth vector bundles over $M$.
For $k\in\mathbb R$, denote by $\Psi^k(M;\Hom(\mathcal E,\mathcal F))$ the class of pseudodifferential
operators of order~$k$ on~$M$ with values in homomorphisms $\mathcal E\to\mathcal F$
and symbols in the class $S^k$; see for instance~\cite[\S 18.1]{ho3}
and~\cite[\S C.1]{DZ}.
These operators act 
\begin{equation}
  \label{e:acting}
C^\infty(M;\mathcal E)\to C^\infty(M;\mathcal F),\quad
\mathcal D'(M;\mathcal E)\to \mathcal D'(M;\mathcal F)
\end{equation}
where $C^\infty(M;\mathcal E)$ denotes the space of smooth sections
and $\mathcal D'(M;\mathcal E)$ denotes the space of distributional sections
\cite[\S 6.3]{ho1}. 
For $k\in\mathbb N_0$, the class $\Psi^k$ includes all smooth differential operators of order $k$.
To each $\mathbf A\in \Psi^k(M;\Hom(\mathcal E;\mathcal F))$ we associate its principal symbol
$$
\sigma(\mathbf A)\in S^k(M;\Hom(\mathcal E;\mathcal F))/S^{k-1}(M;\Hom(\mathcal E;\mathcal F))
$$
and its wavefront set
$
\WF(\mathbf A)\ \subset\ T^*M\setminus 0,
$
which is a closed conic set. Here $T^*M\setminus 0$ denotes the cotangent
bundle of $M$ without the zero section. In the case of $\mathcal E=\mathcal F$ we use the
notation
$
\End(\mathcal E)=\Hom(\mathcal E;\mathcal E).
$
For a distribution $\mathbf u\in \mathcal D'(M;\mathcal E)$, its wavefront
set
$$
\WF(\mathbf u)\subset T^*M\setminus 0
$$
is a closed conic set defined as follows: a point $(x,\xi)\in T^*M\setminus 0$
does \emph{not} lie in~$\WF(\mathbf u)$ if and only if there exists an open conic
neighborhood $U$ of $(x,\xi)$ such that
$\mathbf A \mathbf u\in C^\infty(M;\mathcal E)$ for each
$\mathbf A\in \Psi^k(M;\End(\mathcal E))$ satisfying $\WF(\mathbf A)\subset U$.
See \cite[Theorem 18.1.27]{ho3} for more details. 

{ The above abstract definition is 
useful in this paper but for the reader's convenience we recall the more intuitive local definition in the case of distributions on $ \RR^n $ (see \cite[Definition 8.1.2]{ho1}): if $ u \in \mathcal D' ( \RR^n ) $ and $ ( x, \xi ) 
\in T^* \RR^n \setminus 0=\RR^n \times ( \RR^n \setminus \{0\} ) $ then
\[  ( x, \xi ) \notin \WF ( u ) \ \Longleftrightarrow \ 
\left\{ \begin{array}{c} 
\exists \, \varphi \in C^\infty_{\rm{c}} ( \RR^n ), \ \varphi ( x  ) \neq 0, \ 
 \varepsilon > 0 \ \text{ such that } \\
 | \widehat {\varphi u } (\eta ) | = \mathcal O ( \langle \eta \rangle^{-\infty} )
 \ \text{ for } \ \big| \eta/|\eta | - \xi/|\xi | \big| < \varepsilon . 
\end{array} \right.
\]
Here $ \langle \eta \rangle := ( 1 + |\eta|^2 )^{\frac12} $ and 
$ \mathcal O ( \langle \eta \rangle^{-\infty} ) $ means that the left hand side is 
bounded by $ C_N \langle \eta \rangle^{-N } $ for any $ N $. Since the decay of
the Fourier transform, $ \hat v $, corresponds to regularity of a distribution $ v $, 
this provides ``localized" information both in the position variable $ x $ (thanks to the cutoff $ \varphi $) and in the frequency variable $ \eta $ (thanks to the localization to the cone $ \big| \eta/|\eta| - \xi/|\xi| \big| < \epsilon $).
}

The wavefront set is preserved by
pseudodifferential operators: that is,
\begin{equation}
  \label{e:module-1}
\mathbf A\in \Psi^k(M;\Hom(\mathcal E,\mathcal F)),\
\mathbf u\in \mathcal D'(M;\mathcal E)\ \Longrightarrow\
\WF(\mathbf A\mathbf u)\subset \WF(\mathbf A)\cap\WF(\mathbf u).
\end{equation}
Following~\cite[\S 8.2]{ho1}, for a closed conic set $\Gamma\subset T^*M\setminus 0$ we consider
the space
\begin{equation}
  \label{e:hormander-space}
\mathcal D'_\Gamma(M;\mathcal E)=\{\mathbf u\in \mathcal D'(M;\mathcal E)\colon \WF(\mathbf u)\subset\Gamma\}
\end{equation}
and note that by~\eqref{e:module-1} this space is preserved by pseudodifferential operators.

We also consider the class $\Psi^k_h(M;\Hom(\mathcal E;\mathcal F))$ of semiclassical pseudodifferential
operators with symbols in class $S^k_h$.
The elements of this class are families of operators on~\eqref{e:acting} depending
on a small parameter $h>0$. To each $\mathbf A\in\Psi^k_h(M;\Hom(\mathcal E;\mathcal F))$
correspond its semiclassical principal symbol and wavefront set
$$
\sigma_h(\mathbf A)\in S^k_h(M;\Hom(\mathcal E;\mathcal F))/
hS^{k-1}_h(M;\Hom(\mathcal E;\mathcal F)),\quad
\WFh(\mathbf A)\subset \overline T^*M
$$
where $\overline T^*M$ is the fiber-radially compactified cotangent bundle,
see for instance \cite[\S E.1]{res}. For a tempered $h$-dependent family of distributions
$\mathbf u(h)\in \mathcal D'(M;\mathcal E)$, we can define its wavefront set
$
\WFh(\mathbf u)\subset \overline T^*M.
$

We denote by $\Psi^{\comp}_h(M)\subset \bigcap_k \Psi^k_h(M)$
the class of compactly microlocalized semiclassical pseudodifferential operators,
see \cite[Definition E.29]{res}. 

%%%%%%%%%%%%%%%%%%%%%%%%%%%%%%%%%%%%%%%%%%%%%%%%%%%%%%%%%%%%%%%%%%%%%%%%%%%%%%%%
\subsection{Differential forms}

Let $M$ be a compact oriented manifold.
Denote by $\Omega^k$ the complexified vector bundle of differential $k$-forms
on $M$.
The de Rham cohomology spaces are defined as the quotients
of the spaces of closed forms by the spaces of exact forms, that is
$$
\mathbf H^k(M;\mathbb C)={\displaystyle\{\mathbf u\in C^\infty(M;\Omega^k)\colon d\mathbf u=0\}
\over \displaystyle\{d\mathbf v\colon \mathbf v\in C^\infty(M;\Omega^{k-1})\}}.
$$
These are finite dimensional vector spaces over $\mathbb C$, with the dimensions
\begin{equation}
  \label{e:Betti}
\mathbf b_k(M): =\dim\mathbf H^k(M;\mathbb C)
\end{equation}
called $k$-th \emph{Betti numbers}. (It is convenient for us to study cohomology
over $\mathbb C$, which is of course just the complexification of the cohomology over $\mathbb R$.)

De Rham cohomology is typically formulated in terms of smooth differential forms.
However, the next lemma shows that one can use instead the classes $\mathcal D'_\Gamma$:
%%%%%%%%%%%%%%%%%%%%%%%%%%%%%%%%%%%%%%%%%%%%%%%%%%%%%%%%%%%%%%%%%%%%%%%%%%%%%%%%
\begin{lemm}
  \label{l:de-rham}
Let $\Gamma\subset T^*M\setminus 0$ be a closed conic set.
Using the notation \eqref{e:hormander-space}, assume that
$
\mathbf u\in \mathcal D'_\Gamma(M;\Omega^k)$, 
$d\mathbf u\in C^\infty(M;\Omega^{k+1}).
$

Then there exist
$
\mathbf v\in C^\infty(M;\Omega^k)$ and
$\mathbf w\in \mathcal D'_\Gamma(M;\Omega^{k-1})$ such that 
$\mathbf u=\mathbf v+d\mathbf w.
$
\end{lemm}
%%%%%%%%%%%%%%%%%%%%%%%%%%%%%%%%%%%%%%%%%%%%%%%%%%%%%%%%%%%%%%%%%%%%%%%%%%%%%%%%
\begin{proof}
Fix a smooth Riemannian metric on $M$ { and recall that
it defines an inner product on $ C^\infty ( M ; \Omega^k ) $. Since
$ d : C^\infty ( M ; \Omega^k ) \to C^\infty ( M ; \Omega^{k+1} ) $, we obtain the adjoint operator $ \delta : \mathcal D' ( M ; 
\Omega^{k+1} )  \to \mathcal D' ( M ; \Omega^k ) $.} 
 We use Hodge theory,
in particular the fact that the Hodge Laplacian
$
\mathbf\Delta_k:=d\delta+\delta d:\mathcal D'(M;\Omega^k)\to \mathcal D'(M;\Omega^k)
$
is a second order differential operator with scalar
principal symbol $\sigma(\mathbf\Delta_k)(x,\xi)=|\xi|^2_g$.
By the elliptic parametrix construction (see \cite[Theorem 18.1.24]{ho3}) there exists
a pseudodifferential operator
$
\mathbf Q_k\in \Psi^{-2}(M;\End(\Omega^k))
$
such that
\begin{equation}
  \label{e:paramet}
\mathbf Q_k\mathbf \Delta_k-I, \ \mathbf \Delta_k\mathbf Q_k-I:\mathcal D'(M;\Omega^k)\to C^\infty(M;\Omega^k).
\end{equation}
Using~\eqref{e:module-1} we now take 
$
\mathbf w:=\delta \mathbf Q_k\mathbf u\in\mathcal D'_\Gamma(M;\Omega^{k-1}).
$

Then by~\eqref{e:paramet}
$$
\mathbf u-\delta d \mathbf Q_k\mathbf u-d\mathbf w=\mathbf u-\mathbf \Delta_k\mathbf Q_k \mathbf u\in C^\infty(M;\Omega^k).
$$
Since $d\mathbf u\in C^\infty(M;\Omega^{k+1})$, we have
$$
\mathbf \Delta_{k+1}(d\mathbf Q_k\mathbf u)=d(\mathbf\Delta_k \mathbf Q_k\mathbf u)\in C^\infty(M;\Omega^{k+1}).
$$
By~\eqref{e:paramet} this implies that $d\mathbf Q_k\mathbf u\in C^\infty(M;\Omega^{k+1})$
and thus $\delta d\mathbf Q_k\mathbf u\in C^\infty(M;\Omega^k)$, 
giving $\mathbf v:=\mathbf u-d\mathbf w\in C^\infty(M;\Omega^k)$.
\end{proof}
%%%%%%%%%%%%%%%%%%%%%%%%%%%%%%%%%%%%%%%%%%%%%%%%%%%%%%%%%%%%%%%%%%%%%%%%%%%%%%%%

%%%%%%%%%%%%%%%%%%%%%%%%%%%%%%%%%%%%%%%%%%%%%%%%%%%%%%%%%%%%%%%%%%%%%%%%%%%%%%%%
\subsection{Pollicott--Ruelle resonances}
  \label{s:resonances}

We now follow \cite{fa-sj,DZ} and recall  a microlocal approach to Pollicott--Ruelle resonances.
Let $M$ be a compact manifold and $X$ be a smooth vector field
on $M$. We assume that $e^{tX}$ is an Anosov flow, that is each tangent
space $T_xM$ admits a stable/unstable decomposition
$$
T_xM=\mathbb RX(x)\oplus E_u(x)\oplus E_s(x),\quad
x\in M,
$$
where $E_u(x),E_s(x)$ are subspaces of $T_xM$ depending continuously on $x$
and invariant under the flow and for some constants $C,\nu>0$ and
a fixed smooth metric on $M$,
\begin{equation}
  \label{e:Anosov}
|de^{tX}(x)\cdot v|\leq Ce^{-\nu|t|}\cdot |v|,\quad
\begin{cases}t\geq 0,& v\in E_s(x),\\
t\leq 0,& v\in E_u(x).\end{cases}
\end{equation}
We consider the dual decomposition
$$
T_x^*M=E_0^*(x)\oplus E_u^*(x)\oplus E_s^*(x),
$$
where $E_0^*(x),E_u^*(x),E_s^*(x) $ are dual to $\mathbb RX( x ) ,E_s( x ) ,E_u( x ) $. In particular,
$E_u^*(x)$ is the annihilator of $\mathbb RX(x)\oplus E_u(x)$
and $E_u^* := \bigcup_{ x \in M } E_u^* ( x ) \subset T^*M$ is a closed conic set.

Assume next that $\mathcal E$ is a smooth complex vector bundle over $M$ and
$$
\mathbf P:C^\infty(M;\mathcal E)\to C^\infty(M;\mathcal E)
$$
is a first order differential operator whose principal part is given by $-iX$,
that is
\begin{equation}
  \label{e:principal-part}
\mathbf P(\varphi \mathbf u)=-(iX\varphi)\mathbf u+\varphi(\mathbf P\mathbf u),\quad
\varphi\in C^\infty(M),\quad
\mathbf u\in C^\infty(M;\mathcal E).
\end{equation}
For $\lambda\in\mathbb C$ with sufficiently large $\Im\lambda$, the integral
\begin{equation}
  \label{e:uhp}
\mathbf R(\lambda):=i\int_0^\infty e^{i\lambda t}e^{-it\mathbf P}\,dt:L^2(M;\mathcal E)\to L^2(M;\mathcal E)
\end{equation}
converges and defines a bounded operator on $L^2$, holomorphic in $\lambda$; in fact,
$\mathbf R(\lambda)=(\mathbf P-\lambda)^{-1}$ on $L^2$.

\renewcommand\thefootnote{\dag}% 

The operator $\mathbf R(\lambda)$ admits a meromorphic continuation to the entire complex plane,
\begin{equation}
  \label{e:mercon}
\mathbf R(\lambda):C^\infty(M;\mathcal E)\to \mathcal D'(M;\mathcal E),\quad
\lambda\in\mathbb C,
\end{equation}
and the poles of this meromorphic continuation are the \emph{Pollicott--Ruelle resonances}%
\footnote{To be consistent with \cite{DZ} we use the spectral parameter $\lambda=is$ where
$s$ is the parameter used in~\S1. Note that $\Re s \gg 1 $ corresponds to 
$ \Im \lambda \gg 1 $.}
of the operator $\mathbf P$. See for instance~\cite[\S3.2]{DZ}
and~\cite[Theorems~1.4,1.5]{fa-sj}.

To define the multiplicity of a Pollicott--Ruelle resonance $\lambda_0$,
we use the Laurent expansion of $\mathbf R$ at $\lambda_0$ given by~\cite[Proposition~3.3]{DZ}:
\begin{equation}
\label{eq:proj}
\mathbf R(\lambda)=\mathbf R_H(\lambda)-\sum_{j=1}^{J(\lambda_0)} {(\mathbf P-\lambda_0)^{j-1}\Pi\over (\lambda-\lambda_0)^j},\quad
\mathbf R_H(\lambda),\Pi:\mathcal D'_{E_u^*}(M;\mathcal E)\to\mathcal D'_{E_u^*}(M;\mathcal E),
\end{equation}
where $\mathbf R_H(\lambda)$ is holomorphic at $\lambda_0$,
$\Pi$ is a finite rank operator,
and $\mathcal D'_{E_u^*}(M;\mathcal E)$ is defined using~\eqref{e:hormander-space}.
The fact that $\mathbf R_H(\lambda),\Pi$ can be extended to continuous operators
on $\mathcal D'_{E_u^*}$ follows from
the restrictions on their wavefront sets given in~\cite[(3.7)]{DZ} together with~\cite[Theorem~8.2.13]{ho1}.
The multiplicity of $\lambda_0$, denoted $m_{\mathbf P}(\lambda_0)$, is defined as the dimension of the range of $\Pi$.

The multiplicity of a resonance can be
computed using generalized resonant states.
Here we only need the following special case:
%%%%%%%%%%%%%%%%%%%%%%%%%%%%%%%%%%%%%%%%%%%%%%%%%%%%%%%%%%%%%%%%%%%%%%%%%%%%%%%%
\begin{lemm}
  \label{l:res-states}
Define the space of resonant states at $\lambda_0\in\mathbb C$,
$$
\Res_{\mathbf P}(\lambda_0)=\{\mathbf u\in \mathcal D'_{E_u^*}(M;\mathcal E)\colon
(\mathbf P-\lambda_0)\mathbf u=0\}.
$$
Then $m_{\mathbf P}(\lambda_0)\geq\dim\Res_{\mathbf P}(\lambda_0)$. Moreover
we have $m_{\mathbf P}(\lambda_0)=\dim\Res_{\mathbf P}(\lambda_0)$ under the following
semisimplicity condition:
\begin{equation}
  \label{e:rs-condition}
\mathbf u\in \mathcal D'_{E_u^*}(M;\mathcal E),\quad
(\mathbf P-\lambda_0)^2\mathbf u=0\quad\Longrightarrow\quad
(\mathbf P-\lambda_0)\mathbf u=0.
\end{equation}
\end{lemm}
%%%%%%%%%%%%%%%%%%%%%%%%%%%%%%%%%%%%%%%%%%%%%%%%%%%%%%%%%%%%%%%%%%%%%%%%%%%%%%%%
\begin{proof}
We first assume that~\eqref{e:rs-condition} holds and prove that $m_{\mathbf P}(\lambda_0)\leq\dim\Res_{\mathbf P}(\lambda_0)$.
We have $(\mathbf P-\lambda)\mathbf R(\lambda)=I$ and thus
$(\mathbf P-\lambda_0)^{J(\lambda_0)}\Pi=0$. Take $\mathbf u$ in the range
of $\Pi$, then $\mathbf u\in\mathcal D'_{E_u^*}(M;\mathcal E)$ by
the mapping property in~\eqref{eq:proj} and $(\mathbf P-\lambda_0)^{J(\lambda_0)}\mathbf u=0$.
Arguing by induction using~\eqref{e:rs-condition} we obtain $\mathbf u\in\Res_{\mathbf P}(\lambda_0)$,
finishing the proof.

It remains to show that $\dim\Res_{\mathbf P}(\lambda_0)\leq m_{\mathbf P}(\lambda_0)$.
For that it suffices to prove that
\begin{equation}
  \label{e:reproduce}
\mathbf u\in\Res_{\mathbf P}(\lambda_0)\quad\Longrightarrow\quad \mathbf u=\Pi\mathbf u.
\end{equation}
We recall from~\cite[\S\S3.1,3.2]{DZ} that $\mathbf R(\lambda)$
is the restriction to $C^\infty$ of the inverse of the operator
\begin{equation}
  \label{e:direct}
\mathbf P-\lambda:\{\mathbf v\in H_{sG}(M;\mathcal E)\colon \mathbf P\mathbf v\in H_{sG}(M;\mathcal E)\}
\to H_{sG}(M;\mathcal E),
\end{equation}
where $H_{sG}(M;\mathcal E)\subset\mathcal D'(M;\mathcal E)$ is a specially constructed anisotropic Sobolev space
and we may take any $s>s_0$ where $s_0$ depends on $\lambda$.
Take $s>s_0$ large enough so that $\mathbf u$ lies in the usual Sobolev space $H^{-s}(M;\mathcal E)$.
Since $H_{sG}$ is equivalent to $H^{-s}$
microlocally near $E_u^*$ (see~\cite[(3.3),(3.4)]{DZ}),
we have $\mathbf u\in H_{sG}$.
We compute $(\mathbf P-\lambda)^{-1}\mathbf u=(\lambda_0-\lambda)^{-1}\mathbf u$ for
$\mathbf u\in\Res_{\mathbf P}(\lambda_0)$
and the space $C^\infty$ is dense in $H_{sG}\cap\mathcal D'_{E_u^*}$, thus~\eqref{e:reproduce} follows
from the Laurent expansion~\eqref{eq:proj} applied to $\mathbf u$.
\end{proof}
%%%%%%%%%%%%%%%%%%%%%%%%%%%%%%%%%%%%%%%%%%%%%%%%%%%%%%%%%%%%%%%%%%%%%%%%%%%%%%%%
We finish this section with the following analogue of Rellich's uniqueness theorem in scattering
theory: vanishing of radiation patterns implies rapid decay. { To see 
the connection we refer to the discussion around \cite[(3.6.15)]{res}: 
an outgoing solution $ u = R_0 ( \lambda ) f $, $ R_0 ( \lambda ) 
= ( - \Delta - \lambda^2-i0)^{-1} $, $ f \in C^\infty_{\rm{c}} ( \RR^n ) $, $\lambda>0$, has to have a {\em nonnegative} quantum 
flux $ - \Im \langle ( - \Delta - \lambda^2 ) u , u \rangle = \Im \langle 
R_0 ( \lambda ) f , f \rangle $. If that flux is nonpositive (and thus equal to zero), it follows that
$ u$ is rapidly decaying. In Lemma~\ref{l:mystery} below, the analogue of $ ( - \Delta - \lambda^2 ) u 
$ is $ \mathbf P \mathbf u $ and rapid decay is replaced by smoothness. 
Technically the proof is also different but the commutator argument is 
related to the commutator appearing on the left hand side of \cite[(3.6.15)]{res}.}

\begin{lemm}
\label{l:mystery}
Suppose that there exist
a smooth volume form on $M$ and a smooth inner product on the fibers of $\mathcal E$,
for which $\mathbf P^*=\mathbf P$ on $L^2(M;\mathcal E)$.
Suppose that $ \mathbf u \in \mathcal D'_{E_u^*} ( M;\mathcal E ) $ satisfies
\[ 
\mathbf P \mathbf u \in C^\infty ( M;\mathcal E ) ,\quad \Im\, \langle \mathbf P \mathbf u , \mathbf u \rangle_{L^2} \geq 0 .
\]
Then $ \mathbf u \in C^\infty ( M ;\mathcal E) $. 
\end{lemm}
%%%%%%%%%%%%%%%%%%%%%%%%%%%%%%%%%%%%%%%%%%%%%%%%%%%%%%%%%%%%%%%%%%%%%%%%%%%%%%%%
\Remark Lemma~\ref{l:mystery} applies in particular when $\mathbf u$
is a resonant state at some $\lambda\in\mathbb R$ (replacing $\mathbf P$ by $\mathbf P-\lambda$),
showing that all such resonant states are smooth. This represents a borderline case
since for $\Im\lambda>0$ the integral~\eqref{e:uhp} converges and
thus there are no resonances.
%%%%%%%%%%%%%%%%%%%%%%%%%%%%%%%%%%%%%%%%%%%%%%%%%%%%%%%%%%%%%%%%%%%%%%%%%%%%%%%%
\begin{proof}
We introduce the semiclassical parameter $h>0$ and use the following statement
relating semiclassical and nonsemiclassical wavefront sets of an $h$-independent
distribution $\mathbf v$, see~\cite[(2.6)]{DZ}:
\begin{equation}
  \label{e:wfrel}
\WF(\mathbf v)=\WFh(\mathbf v)\cap (T^*M\setminus 0).
\end{equation}
Since $\mathbf u\in \mathcal D'_{E_u^*}$ and $\mathbf P\mathbf u\in C^\infty$ we have
\begin{equation}
\label{e:wf-1}
  \WFh(\mathbf u)\cap (T^*M\setminus 0)\ \subset\ E_u^*,\quad
  \WFh(\mathbf P\mathbf u)\cap (\overline T^*M\setminus 0)\ =\ \emptyset.
\end{equation}
(The last statement uses the fiber-radially compactified cotangent bundle
and it follows immediately from the proof of~\cite[(2.6)]{DZ} in~\cite[\S C.2]{DZ}.)

It suffices to prove that for each 
$A\in \Psi^{\comp}_h(M)$ with $\WFh(A)\subset T^*M\setminus 0$,
there exists $B\in\Psi^{\comp}_h(M)$ with $\WFh(B)\subset T^*M\setminus 0$
such that
\begin{equation}
  \label{e:mystery-step}
\|A\mathbf u\|_{L^2}\leq Ch^{1/2}\|B\mathbf u\|_{L^2}+\mathcal O(h^\infty).
\end{equation}
Indeed, fix $N>0$ such that $\mathbf u\in H^{-N}$, then $\|A\mathbf u\|_{L^2}\leq Ch^{-N}$
for all $A\in\Psi^{\comp}_h(M)$. By induction~\eqref{e:mystery-step} implies that
$\|A\mathbf u\|_{L^2}=\mathcal O(h^\infty)$. This gives $\WFh(\mathbf u)\cap (T^*M\setminus 0)=\emptyset$
and thus by~\eqref{e:wfrel} $\WF(\mathbf u)=\emptyset$, that is  $\mathbf u\in C^\infty$.

To show~\eqref{e:mystery-step}, note that $h\mathbf P\in\Psi^1_h(M;\End(\mathcal E))$ and
its principal symbol is scalar and given by
$$
\sigma_h(h\mathbf P)=p,\quad
p(x,\xi)=\langle \xi,X(x)\rangle.
$$
We now claim that there exists $ \chi\in C_c^\infty(T^*M;[0,1]) $ such that
$$
\begin{gathered}
\supp(1-\chi)\subset T^*M\setminus 0,\quad
H_p\chi \leq  0 \text{ near }E_u^*,\quad
H_p\chi <0  \text{ on }E_u^*\cap \WFh(A).
\end{gathered}
$$
To construct $\chi$, we first use
part~2 of~\cite[Lemma~C.1]{DZ} (applied to $L:=E_u^*$ which is
a radial source for $-p$) to construct
$f_1\in C^\infty(T^*M\setminus 0;[0,\infty))$ homogeneous of degree~1,
satisfying $f_1(x,\xi)\geq c|\xi|$ and $H_pf_1\geq cf_1$ in a conic neighborhood
of $E_u^*$, for some $c>0$. Next we put $\chi:=\chi_1\circ f_1$ where
$ \chi_1\in C_c^\infty(\mathbb R;[0,1]) $ satisfies
$$
\begin{aligned}
\chi_1=1\text{ near }0,\quad
\chi'_1\leq 0 \text{ on }[0,\infty), \quad
\chi'_1<0 \text{ on }f_1(\WFh(A)).
\end{aligned}
$$
It is then straightforward to see that $\chi$ has the required properties.

We now quantize $\chi$ to obtain an operator
$$
F\in \Psi^{\comp}_h(M),\quad
\sigma_h(F)=\chi,\quad
\WFh(I-F)\subset \overline T^*M\setminus 0,\quad
F^*=F.
$$
Since $H_p\chi\leq 0$ near $E_u^*$ and $H_p\chi<0$ on $E_u^*\cap \WFh(A)$ there exists
\begin{gather*}
A_1\in\Psi^{\comp}_h(M), \ \ \WFh(A_1)\subset T^*M\setminus 0, \ \ 
\WFh(A_1)\cap E_u^*=\emptyset,  
\end{gather*}
such that 
\begin{equation}
\label{eq:ineq}
-{\textstyle{ 1\over 2}}H_p\chi +|\sigma_h(A_1)|^2\geq C^{-1}|\sigma_h(A)|^2
\end{equation}
where $C>0$ is some constant.

Fix $B\in\Psi^{\comp}_h(M)$ with 
$ \WFh ( B ) \subset T^* M \setminus 0 $ so  that
\begin{equation}
\label{eq:defB}
\left( \WFh\big([\mathbf P,F]\big)\cup\WFh(A_1)\cup\WFh(A) \right) \cap 
\WFh ( I - B ) = \emptyset .
\end{equation}
By the second part of \eqref{e:wf-1} we have
$(I-F)\mathbf P\mathbf u=\mathcal O(h^\infty)_{C^\infty}$.
Since $\Im\langle \mathbf P\mathbf u,\mathbf u\rangle_{L^2}\geq 0$ this gives
\begin{equation}
  \label{e:varnish}
-\Im\langle F\mathbf P\mathbf u,\mathbf u\rangle_{L^2}\leq \mathcal O(h^\infty).
\end{equation}
On the other hand, since $\mathbf P$ and $F$ are both symmetric, we get
\begin{equation}
  \label{e:varnish2}
-\Im\langle F\mathbf P\mathbf u,\mathbf u\rangle_{L^2}={\textstyle{1\over 2i}}\langle [\mathbf P,F]\mathbf u,\mathbf u\rangle_{L^2}.
\end{equation}
We now observe that
$$
{\textstyle{1\over 2i}}[\mathbf P,F]\in \Psi^{\comp}_h(M;\mathcal E),\quad
\sigma_h\left({\textstyle{1\over 2i}}[\mathbf P,F]\right)=-{\textstyle{1\over 2}}H_p\chi.
$$
Using \eqref{eq:ineq} we can apply the sharp G\r arding inequality 
(see for instance \cite[Theorem 9.11]{ev-zw}) 
to the operator ${1\over 2i}[\mathbf P,F]+A_1^*A_1-C^{-1}A^*A$
and the section $B\mathbf u$
to obtain
$$
\|A B \mathbf u\|_{L^2}^2\leq C\|A_1 B \mathbf u\|_{L^2}^2+{\textstyle{C\over 2i}}\langle B^*[\mathbf P,F]B\mathbf u,\mathbf u\rangle_{L^2}
+Ch\|B\mathbf u\|_{L^2}^2.
$$
From \eqref{eq:defB} we see that 
$ A B \mathbf u \equiv A \mathbf u $, 
$ A_1 B\mathbf u \equiv A_1 \mathbf  u $ and 
$  B^* [\mathbf P,F] B\mathbf u \equiv [\mathbf P,F] \mathbf u $, modulo $ \mathcal O ( h^\infty )_{C^\infty}$. 
Also, the first part of \eqref{e:wf-1} shows that $A_1\mathbf u=\mathcal O(h^\infty)_{L^2}$. Using~\eqref{e:varnish}
and~\eqref{e:varnish2}
we obtain~\eqref{e:mystery-step}, finishing the proof.
\end{proof}
%%%%%%%%%%%%%%%%%%%%%%%%%%%%%%%%%%%%%%%%%%%%%%%%%%%%%%%%%%%%%%%%%%%%%%%%%%%%%%%%

%%%%%%%%%%%%%%%%%%%%%%%%%%%%%%%%%%%%%%%%%%%%%%%%%%%%%%%%%%%%%%%%%%%%%%%%%%%%%%%%
\subsection{Contact flows and geodesic flows}
\label{s:gf}

Assume that $M$ is a compact three dimensional manifold and 
$\alpha\in C^\infty(M;\Omega^1)$ is a contact form, that is
\[
d\vol_M:=
\alpha\wedge d\alpha\neq 0\quad\text{everywhere.} 
\]
Then $d\vol_M$ fixes a volume form and an orientation on $M$.
The form $\alpha$ determines uniquely the \emph{Reeb vector field} $X$
on $M$ satisfying the conditions (with $\iota$ denoting the interior product)
\begin{equation}
\label{eq:dReeb}
\iota_X\alpha=1 \,,\quad 
\iota_X (d\alpha)=0. 
\end{equation}
We record for future use the following identity
which can be checked by applying both sides to a frame containing $X$:
\begin{equation}
  \label{e:identity}
\mathbf u\wedge d\alpha=(\iota_X \mathbf u)\,d\vol_M\quad\text{for all }
\mathbf u\in\mathcal D'(M;\Omega^1).
\end{equation}
We now consider the special case when $M$ is the unit cotangent bundle
of a compact Riemannian surface $(\Sigma,g)$:
\begin{equation}
  \label{e:circle-bundle}
M=S^*\Sigma=\{(x,\xi)\in T^*\Sigma\colon |\xi|_g=1\}.
\end{equation}
Let $ j: S^* \Sigma \hookrightarrow T^* \Sigma $ and put $ \alpha := j^* ( \xi dx ) $.
Then $\alpha$ is a contact form and the corresponding vector field $X$ is the generator
of the geodesic flow.

We recall a standard topological fact which 
will be used in passing from the Betti number of $ M = S^* \Sigma $ to 
the Euler characteristic of $ \Sigma $. It is an immediate consequence
of the Gysin long exact sequence; we provide a direct proof for the reader's convenience:
%%%%%%%%%%%%%%%%%%%%%%%%%%%%%%%%%%%%%%%%%%%%%%%%%%%%%%%%%%%%%%%%%%%%%%%%%%%%%%%%
\begin{lemm}
  \label{l:topo}
Assume that $(\Sigma,g)$ is a compact connected oriented Riemannian surface of nonzero Euler characteristic,
$M$ is given by~\eqref{e:circle-bundle}, and $\pi:M\to\Sigma$ is the projection map.
Then for any 
$
\mathbf u\in C^\infty(M;\Omega^1)$ with 
$ d\mathbf u=0$ 
there exist $\mathbf v,\varphi$ such that
\begin{equation}
  \label{e:topo}
\mathbf u=\pi^* \mathbf v+d\varphi,\quad
\mathbf v\in C^\infty(\Sigma;\Omega^1),\quad
d\mathbf v=0,\quad
\varphi\in C^\infty(M).
\end{equation}
In particular, $\pi^*:\mathbf H^1(\Sigma;\mathbb C)\to \mathbf H^1(M;\mathbb C)$ is an isomorphism.
\end{lemm}
%%%%%%%%%%%%%%%%%%%%%%%%%%%%%%%%%%%%%%%%%%%%%%%%%%%%%%%%%%%%%%%%%%%%%%%%%%%%%%%%
\begin{proof}
For computations below, we will use positively oriented local coordinates $(x_1,x_2)$ on $\Sigma$
in which the metric has the form
$g=e^{2\psi}(dx_1^2+dx_2^2)$,
for some smooth real-valued function $\psi$. The corresponding coordinates
on $M$ are $(x_1,x_2,\theta)$ with the covector
given by $\xi=e^{\psi}(\cos\theta,\sin\theta)$.
Let $V$ be the vector field on $M$ which generates rotations in the fibers of $\pi$.
In local coordinates, we have $V=\partial_\theta$.
To show~\eqref{e:topo} it suffices to find $\varphi\in C^\infty(M)$ such that
\begin{equation}
  \label{e:topo-p1}
V\varphi=\iota_V\mathbf u.
\end{equation}
Indeed, put $\mathbf w:=\mathbf u-d\varphi$. Then $d\mathbf w=0$
and $\iota_V\mathbf w=0$. A calculation in local coordinates
shows that $\mathbf w=\pi^*\mathbf v$
for some $\mathbf v\in C^\infty(\Sigma;\Omega^1)$
such that $d\mathbf v=0$.

A smooth solution to~\eqref{e:topo-p1} exists if
$\mathbf u$ integrates to 0 on each fiber of $\pi$.
Since~$\mathbf u$ is closed
and all fibers are homotopic to each other, the integral
of $\mathbf u$ along each fiber is equal to some constant $c\in\mathbb C$,
thus it remains to show that $c=0$.

Let $K\in C^\infty(\Sigma)$ be the Gaussian curvature of $\Sigma$ and
$d\vol_\Sigma\in C^\infty(\Sigma;\Omega^2)$ the volume form
of $(\Sigma,g)$, written in local coordinates as
$d\vol_\Sigma=e^{2\psi}dx_1\wedge dx_2$.
With $\chi(\Sigma)\neq 0$ denoting the Euler characteristic of $\Sigma$, we have
by Gauss--Bonnet theorem
$$
\int_M\mathbf u\wedge \pi^*(K\,d\vol_\Sigma)=2\pi \chi(\Sigma)\cdot c.
$$
It then remains to prove that $
\int_M \mathbf u\wedge \pi^* (K\, d\vol_\Sigma)=0$.
This follows via integration by parts from the identity
$\pi^*(K\,d\vol_\Sigma)=-dV^*$,
where $V^*\in C^\infty(M;\Omega^1)$ is the connection form, namely the unique
1-form satisfying the relations
$$
\iota_V V^*=1,\quad
d\alpha=V^*\wedge \beta,\quad
d\beta=-V^*\wedge \alpha,
$$
where $\alpha$ is the contact form 
and $\beta$ is the pullback of $\alpha$
by the $\pi/2$ rotation in the fibers of $\pi$. This can be checked
in local coordinates using the formulas $\alpha=e^\psi(\cos\theta\,dx_1+\sin\theta\,dx_2)$,
$\beta=e^{\psi}(-\sin\theta\,dx_1+\cos\theta\,dx_2)$,
$V^*=\partial_{x_1}\psi\,dx_2-\partial_{x_2}\psi\,dx_1+d\theta$,
$K=-e^{-2\psi}\Delta\psi$; see also~\cite[\S3]{g-k}.

Having established~\eqref{e:topo}, we see immediately
that $\pi^*:\mathbf H^1(\Sigma;\mathbb C)\to\mathbf H^1(M;\mathbb C)$
is onto. To show that $\pi^*$ is one-to-one, assume
that $\mathbf v\in C^\infty(\Sigma;\Omega^1)$ satisfies
$\pi^*\mathbf v=d\varphi$ for some $\varphi\in C^\infty(M)$.
Then $V\varphi=\iota_Vd\varphi=0$, therefore $\varphi=\pi^* \chi$
for some $\chi\in C^\infty(\Sigma)$ and $\mathbf v=d\chi$ is exact.
\end{proof}
%%%%%%%%%%%%%%%%%%%%%%%%%%%%%%%%%%%%%%%%%%%%%%%%%%%%%%%%%%%%%%%%%%%%%%%%%%%%%%%%
%%%%%%%%%%%%%%%%%%%%%%%%%%%%%%%%%%%%%%%%%%%%%%%%%%%%%%%%%%%%%%%%%%%%%%%%%%%%%%%%

%%%%%%%%%%%%%%%%%%%%%%%%%%%%%%%%%%%%%%%%%%%%%%%%%%%%%%%%%%%%%%%%%%%%%%%%%%%%%%%%
\section{Proof}
\label{proof}

In this section we prove the main theorem in a slightly more general 
setting~-- see Proposition \ref{l:main}. We assume throughout that
$M$ is a three-dimensional connected compact manifold, $\alpha$ is a contact form on $M$, and
$X$ is the Reeb vector field of $\alpha$ generating an Anosov flow (see~\S\S\ref{s:resonances},\ref{s:gf}).
For the application to zeta functions we also assume that the corresponding 
stable/unstable bundles $E_u,E_s$ are orientable.

%%%%%%%%%%%%%%%%%%%%%%%%%%%%%%%%%%%%%%%%%%%%%%%%%%%%%%%%%%%%%%%%%%%%%%%%%%%%%%%%
\subsection{Zeta function and Pollicott--Ruelle resonances}
\label{s:zeta}

For $k=0,1,2$, let $\Omega^k_0\subset\Omega^k$ be the bundle of exterior $k$-forms $\mathbf u$
on $M$ such that $\iota_X \mathbf u=0$.
Consider the following operator satisfying~\eqref{e:principal-part}:
$$
\mathbf P_k:=-i\mathcal L_X:\mathcal D'(M;\Omega^k_0)\to\mathcal D'(M;\Omega^k_0).
$$
Note that by Cartan's formula
$$
\mathbf P_k\mathbf u=-i\,\iota_X(d\mathbf u),\quad
\mathbf u\in\mathcal D'(M;\Omega^k_0).
$$
As discussed in~\S\ref{s:resonances} we may consider Pollicott--Ruelle resonances
associated to the operators $\mathbf P_k$, denoting their multiplicities
as follows:
$$
m_k(\lambda):=m_{\mathbf P_k}(\lambda)\in\mathbb N_0,\quad
\lambda\in\mathbb C.
$$
The connection with the Ruelle zeta function comes from the following standard
formula (see \cite[(2.5) and~\S 4]{DZ})
for the meromorphic continuation of $\zeta_R$:
$$
\zeta_R(s)={\zeta_1(s )\over \zeta_0( s )\zeta_2(s  )},\quad
s\in\mathbb C.
$$
(It is here that we the assumption that the stable and unstable bundle are orientable.)
Here each $\zeta_k(s )$ is an entire function having a zero of multiplicity
$m_k( i s )$ at each $ s \in\mathbb C$. Therefore, $\zeta_R(s )$ has a zero at
$s =0$ of multiplicity
\begin{equation}
  \label{e:mul-ultimate}
m_R(0):=m_1(0)-m_0(0)-m_2(0).
\end{equation}
By Lemma~\ref{l:res-states} the multiplicity $m_k(0)$ can be calculated as
\begin{equation}
  \label{e:dim-mult}
m_k(0)=\dim\Res_k(0),
\end{equation}
where $\Res_k(0)$ is the space of resonant states at zero,
\begin{equation}
\label{eq:Resk}
\Res_k(0)=\{\mathbf u\in \mathcal D'_{E_u^*}(M;\Omega^k)\colon \iota_X\mathbf u=0,\
\iota_X(d\mathbf u)=0\}
\end{equation}
provided that the semisimplicity condition~\eqref{e:rs-condition} is satisfied:
\begin{equation}
  \label{e:algsim}
\mathbf u\in \mathcal D'_{E_u^*}(M;\Omega^k),\quad
\iota_X\mathbf u=0,\quad
\iota_X(d\mathbf u)\in \Res_k(0)\quad\Longrightarrow\quad
\iota_X(d\mathbf u)=0.
\end{equation}
The main result of this section is
%%%%%%%%%%%%%%%%%%%%%%%%%%%%%%%%%%%%%%%%%%%%%%%%%%%%%%%%%%%%%%%%%%%%%%%%%%%%%%%%
\begin{prop}
\label{l:main}
In the notation of \eqref{eq:Resk} we have
\begin{enumerate}
\item $\dim\Res_0(0)=\dim\Res_2(0)=1$;
\item $\dim\Res_1(0)$ is equal to the Betti number $\mathbf b_1(M)$ defined in~\eqref{e:Betti};
\item the condition~\eqref{e:algsim} holds for $k=0,1,2$.
\end{enumerate}
\end{prop}
%%%%%%%%%%%%%%%%%%%%%%%%%%%%%%%%%%%%%%%%%%%%%%%%%%%%%%%%%%%%%%%%%%%%%%%%%%%%%%%%
It is direct to see that Proposition~\ref{l:main} implies the main theorem when
applied to the case $M=S^*\Sigma$ discussed in~\S\ref{s:gf} (strictly speaking,
to each connected component of $\Sigma$). Indeed, $X$ generates an Anosov flow
since $\Sigma$ is negatively curved (see for example~\cite[Theorem~3.9.1]{k}),
the stable/unstable bundles are orientable since $\Sigma$ is orientable and $m_R(0)=\mathbf b_1(M)-2$ equals to $-\chi(\Sigma)$ by Lemma~\ref{l:topo}.

%%%%%%%%%%%%%%%%%%%%%%%%%%%%%%%%%%%%%%%%%%%%%%%%%%%%%%%%%%%%%%%%%%%%%%%%%%%%%%%%
\subsection{Scalars and 2-forms}

We start the proof of Proposition~\ref{l:main} by considering the cases $k=0$
and $k=2$:
%%%%%%%%%%%%%%%%%%%%%%%%%%%%%%%%%%%%%%%%%%%%%%%%%%%%%%%%%%%%%%%%%%%%%%%%%%%%%%%%
\begin{lemm}
  \label{l:scalars}
We have
\begin{equation}
\Res_0(0)=\{c\colon c\in\mathbb C\},\quad
\Res_2(0)=\{c\,d\alpha\colon c\in\mathbb C\},
\end{equation}
and~\eqref{e:algsim} holds for $k=0,2$, that is the resonance at 0 for 
$ k=0,2$ is simple. 
\end{lemm}
%%%%%%%%%%%%%%%%%%%%%%%%%%%%%%%%%%%%%%%%%%%%%%%%%%%%%%%%%%%%%%%%%%%%%%%%%%%%%%%%
\Remark The argument for $\Res_0(0)$ applies to any contact Anosov flow on
a compact connected manifold. It can be generalized
to show that $\Res_0(0)$ consists of constant functions and $\Res_0(\lambda)$ is trivial
for all $\lambda\in\mathbb R\setminus 0$. This in particular
implies that the flow is mixing.
%%%%%%%%%%%%%%%%%%%%%%%%%%%%%%%%%%%%%%%%%%%%%%%%%%%%%%%%%%%%%%%%%%%%%%%%%%%%%%%%
\begin{proof}
We first handle the case of $\Res_0(0)$. Clearly this space contains constant functions,
therefore we need to show that
\begin{equation}
  \label{e:const}
u\in\mathcal D'_{E_u^*}(M),\quad
Xu=0\quad\Longrightarrow\quad u=c\quad\text{for some }c\in\mathbb C.
\end{equation}
By Lemma~\ref{l:mystery} we have $u\in C^\infty(M)$.
Since $Xu=0$ we have $u=u\circ e^{tX}$ and thus
$$
\langle du(x),v\rangle=\langle du(e^{tX}(x)),de^{tX}(x)\cdot v\rangle\quad\text{for all }
(x,v)\in TM,\ t\in\mathbb R.
$$
Now if $v\in E_s(x)$ then taking the limit as $t\to \infty$ and using~\eqref{e:Anosov}
we obtain
$\langle du(x),v\rangle=0$. Similarly if $v\in E_u(x)$ then
the limit $t\to -\infty$ gives
$\langle du(x),v\rangle=0$. Therefore $du|_{E_u\oplus E_s}=0$. However
$E_u\oplus E_s=\ker \alpha$, thus we have for some $\varphi\in C^\infty(M)$,
$$
du=\varphi\alpha.
$$
We have $0=\alpha\wedge d(\varphi\alpha)=\varphi\alpha\wedge d\alpha$,
thus $du=0$, implying~\eqref{e:const} since $M$ is connected.

Next, \eqref{e:algsim} holds for $k=0$. Indeed, if $u\in\mathcal D'_{E_u^*}(M)$
then
$$
\int_M (Xu)\,d\vol_M=0,
$$
implying that $Xu$ cannot be a nonzero element of $\Res_0(0)$.

Now, assume that $\mathbf u\in \mathcal D'_{E_u^*}(M;\Omega^2)$
satisfies $\iota_X\mathbf u=0$. Then $\mathbf u$ can be written as
$$
\mathbf u=u\,d\alpha, \quad
u\in \mathcal D'_{E_u^*}(M);\quad
\iota_X(d\mathbf u)=(Xu)d\alpha.
$$
Therefore the case of $\Res_2(0)$ follows immediately from that of $\Res_0(0)$.
\end{proof}
%%%%%%%%%%%%%%%%%%%%%%%%%%%%%%%%%%%%%%%%%%%%%%%%%%%%%%%%%%%%%%%%%%%%%%%%%%%%%%%%
Lemma~\ref{l:scalars} implies solvability
of the equation $Xu=f$ in the class $\mathcal D'_{E_u^*}$:
%%%%%%%%%%%%%%%%%%%%%%%%%%%%%%%%%%%%%%%%%%%%%%%%%%%%%%%%%%%%%%%%%%%%%%%%%%%%%%%%
\begin{prop}
  \label{l:scalar-invertibility}
Assume that $f\in C^\infty(M)$ and $\int_M f\,d\vol_M=0$. Then
there exists $u\in \mathcal D'_{E_u^*}(M)$ such that
$Xu=f$.
\end{prop}
%%%%%%%%%%%%%%%%%%%%%%%%%%%%%%%%%%%%%%%%%%%%%%%%%%%%%%%%%%%%%%%%%%%%%%%%%%%%%%%%
\begin{proof}
It follows from Lemma~\ref{l:scalars} and the proof of Lemma~\ref{l:res-states} that the 
resolvent $\mathbf R_0(\lambda)$ of the operator $\mathbf P_0=-iX$ defined in~\eqref{e:mercon} has the expansion
$$
\mathbf R_0(\lambda)=\mathbf R_H(\lambda)-{\Pi\over\lambda}
$$
where $\mathbf R_H(\lambda)$ is holomorphic at $\lambda=0$ and the range of $ \Pi $ consists of constant functions.
By analytic continuation from~\eqref{e:uhp}, we see that $\mathbf R_0(\lambda)^*=-\mathbf R_{-\mathbf P_0}(-\bar\lambda)$
where $\mathbf R_{-\mathbf P_0}(\lambda)$ is the resolvent of $-\mathbf P_0$. Applying
Lemma~\ref{l:scalars} to the field $-X$, we see that the range of $\Pi^*$ also consists
of constant functions. By~\eqref{e:reproduce} we have $\Pi(1)=1$, therefore
$$
\Pi u={1\over\vol(M)}\int_M u\,d\vol_M. 
$$
Now, put
$u:=-i\mathbf R_H(0)f$, then $u\in\mathcal D'_{E_u^*}(M)$ by~\eqref{eq:proj}.
Since $\Pi f=0$ and $(\mathbf P_0-\lambda)\mathbf R_0(\lambda)=I$, we have
$Xu=f$.
\end{proof}
%%%%%%%%%%%%%%%%%%%%%%%%%%%%%%%%%%%%%%%%%%%%%%%%%%%%%%%%%%%%%%%%%%%%%%%%%%%%%%%%

%%%%%%%%%%%%%%%%%%%%%%%%%%%%%%%%%%%%%%%%%%%%%%%%%%%%%%%%%%%%%%%%%%%%%%%%%%%%%%%%
\subsection{1-forms} \label{s:1fo}

It remains to show Proposition~\ref{l:main} for the case $k=1$,
that is to analyse the space
$$
\Res_1(0)=\{\mathbf u\in \mathcal D'_{E_u^*}(M,\Omega^1)\colon
\iota_X\mathbf u=0,\
\iota_X(d\mathbf u)=0\}.
$$
The next lemma shows that the $\dim\Res_1(0)=\mathbf b_1(M)$:
%%%%%%%%%%%%%%%%%%%%%%%%%%%%%%%%%%%%%%%%%%%%%%%%%%%%%%%%%%%%%%%%%%%%%%%%%%%%%%%%
\begin{lemm}
  \label{l:1-forms}
Assume that $\mathbf u\in \Res_1(0)$. Then there exists $\varphi\in \mathcal D'_{E_u^*}(M)$
such that
\begin{equation}
  \label{e:de-rham-iso-0}
\mathbf u-d\varphi\in C^\infty(M;\Omega^1),\quad
d(\mathbf u-d\varphi)=0.
\end{equation}
The cohomology class $[\mathbf u-d\varphi]\in \mathbf H^1(M;\mathbb C)$
is independent of the choice of $\varphi$. The map
\begin{equation}
  \label{e:de-rham-iso-1}
\Res_1(0)\ni\mathbf u\ \mapsto\ [\mathbf u-d\varphi]\in \mathbf H^1(M;\mathbb C)
\end{equation}
is a linear isomorphism.
\end{lemm}
%%%%%%%%%%%%%%%%%%%%%%%%%%%%%%%%%%%%%%%%%%%%%%%%%%%%%%%%%%%%%%%%%%%%%%%%%%%%%%%%
\begin{proof}
We first show that
\begin{equation}
  \label{e:it-is-closed}
\mathbf u\in\Res_1(0)\ \Longrightarrow\ d\mathbf u=0.
\end{equation}
Definition~\eqref{eq:Resk} shows that $ d\mathbf u\in \Res_2(0)$ and 
therefore by Lemma~\ref{l:scalars} we have
$d\mathbf u=c\,d\alpha$ for some $c\in\mathbb C$.
From~\eqref{e:identity} and $\iota_X\mathbf u=0$ we also have $\mathbf u\wedge d\alpha=0$,
thus Stokes's theorem gives~\eqref{e:it-is-closed}:
$$
c\vol(M)=\int_M \alpha\wedge d\mathbf u = 
 \int_M \mathbf u\wedge d\alpha
=0.
$$
Lemma~\ref{l:de-rham} and \eqref{e:it-is-closed} imply the existence
of $\varphi\in\mathcal D'_{E_u^*}(M)$ such that~\eqref{e:de-rham-iso-0} holds.
Moreover, if $\tilde \varphi\in \mathcal D'_{E_u^*}(M)$ is another function satisfying~\eqref{e:de-rham-iso-0}
then $d(\varphi-\tilde \varphi)\in C^\infty(M;\Omega^1)$, implying by Lemma~\ref{l:de-rham} that
$\varphi-\tilde \varphi\in C^\infty(M)$. Therefore $\mathbf u-d\varphi$ and $\mathbf u-d\tilde \varphi$
belong to the same de Rham cohomology class, thus the map~\eqref{e:de-rham-iso-1} is well-defined.

Next, assume that~\eqref{e:de-rham-iso-0} holds and
$\mathbf u-d\varphi$ is exact. By changing $\varphi$ we may assume that
$\mathbf u=d\varphi$. Since $\iota_X\mathbf u=0$ we have $X\varphi=0$,
which by Lemma~\ref{l:scalars} implies that $\varphi$ is constant
and thus $\mathbf u=0$. This shows that~\eqref{e:de-rham-iso-1}
is injective.

It remains to show that~\eqref{e:de-rham-iso-1} is surjective.
For that, take a closed $\mathbf v\in C^\infty(M;\Omega^1)$. We need to find $\varphi\in\mathcal D'_{E_u^*}(M)$
such that $\mathbf v+d\varphi\in \Res_1(0)$. This
is equivalent to $\iota_X(\mathbf v+d\varphi)=0$, that is
$
X\varphi=-\iota_X\mathbf v.
$
A solution $\varphi$ to the above equation exists by Proposition~\ref{l:scalar-invertibility}
since \eqref{e:identity} implies
$$
\int_M \iota_X\mathbf v\,d\vol_M= \int_M \mathbf v\wedge d\alpha=\int_M \alpha\wedge d\mathbf v=0.
$$
This finishes the proof.
\end{proof}
%%%%%%%%%%%%%%%%%%%%%%%%%%%%%%%%%%%%%%%%%%%%%%%%%%%%%%%%%%%%%%%%%%%%%%%%%%%%%%%%
To prove Proposition~\ref{l:main} it remains to show the
semisimplicity condition:
%%%%%%%%%%%%%%%%%%%%%%%%%%%%%%%%%%%%%%%%%%%%%%%%%%%%%%%%%%%%%%%%%%%%%%%%%%%%%%%%
\begin{lemm}
  \label{l:a-s}
Suppose that 
$$ \mathbf u\in\mathcal D'_{E_u^*}(M;\Omega^1),\quad
\iota_X\mathbf u=0,\quad
\iota_X(d\mathbf u)=\mathbf v\in \Res_1(0).
$$
Then $\mathbf v=0$, that is, condition \eqref{e:algsim} holds for $k=1$.
\end{lemm}
%%%%%%%%%%%%%%%%%%%%%%%%%%%%%%%%%%%%%%%%%%%%%%%%%%%%%%%%%%%%%%%%%%%%%%%%%%%%%%%%
\begin{proof}
We have $\alpha\wedge d\mathbf u=a\,d\vol_M$
for some $a\in \mathcal D'_{E_u^*}(M)$. By~\eqref{e:identity},
$$
\int_M a\,d\vol_M=\int_M \mathbf u\wedge d\alpha =\int_M \iota_X\mathbf u\,d\vol_M=0.
$$
Moreover since $\mathcal L_X(\alpha)=0$, $\mathcal L_X(d\alpha)=0$, and $d\mathbf v=0$ by~\eqref{e:it-is-closed},
we have
$$
(Xa)\,d\vol_M=\mathcal L_X(\alpha\wedge d\mathbf u)=\alpha\wedge d\mathbf v=0.
$$
Thus $Xa=0$ and Lemma~\ref{l:scalars} gives that $a=0$ and thus
$\alpha\wedge d\mathbf u=0$. This implies
$d\mathbf u= \alpha \wedge \iota_X d \mathbf u = \alpha\wedge \mathbf v$.
Now by Lemma~\ref{l:1-forms} there exist
$$
\mathbf w\in C^\infty(M;\Omega^1),\quad
\varphi\in \mathcal D'_{E_u^*}(M),\quad
\mathbf v=\mathbf w+d\varphi,\quad
d\mathbf w=0.
$$
Since $\iota_X\mathbf v=0$ we have $X\varphi=-\iota_X\mathbf w$.
Integration by parts together with~\eqref{e:identity} gives
\begin{equation}
\label{eq:zeroprod}
\begin{aligned}
0&=\Re \int_M d\mathbf u\wedge\overline{\mathbf w}
=\Re \int_M\alpha\wedge d\varphi\wedge\overline{\mathbf w}\\
&=\Re \int_M\varphi \,\overline{\mathbf w}\wedge d\alpha
=-\Re \langle X\varphi,\varphi\rangle_{L^2}.
\end{aligned}
\end{equation}
By Lemma~\ref{l:mystery} with $\mathbf P=-iX$ this implies $\varphi\in C^\infty(M)$
and thus $\mathbf v\in C^\infty(M;\Omega^1)$.

We can now use the same argument as in the proof of Lemma \ref{l:scalars}:
$ ( e^{tX } )^* \mathbf v = \mathbf v $ and hence 
\[  \langle \mathbf v (x ) , z \rangle = \langle 
\mathbf v ( e^{tX } x ) , d e^{tX } ( x ) \cdot z \rangle , \ \ ( x, z ) \in T M, \ \ t \in \RR . \]
The right hand side tends to zero as $ t \to \infty $ for $ z \in E_s ( x ) $, 
and as $ t \to - \infty $ for $ z \in E_u ( x ) $. Since $ \iota_X \mathbf v = 0 $
it follows that $ \mathbf v = 0 $.
\end{proof}
%%%%%%%%%%%%%%%%%%%%%%%%%%%%%%%%%%%%%%%%%%%%%%%%%%%%%%%%%%%%%%%%%%%%%%%%%%%%%%%%

%%%%%%%%%%%%%%%%%%%%%%%%%%%%%%%%%%%%%%%%%%%%%%%%%%%%%%%%%%%%%%%%%%%%%%%%%%%%%%%%
%%%%%%%%%%%%%%%%%%%%%%%%%%%%%%%%%%%%%%%%%%%%%%%%%%%%%%%%%%%%%%%%%%%%%%%%%%%%%%%%

\end{document}